\documentclass[11pt,reqno]{amsart}
\usepackage{amssymb,latexsym}
\usepackage{graphicx}
\usepackage{comment}
\usepackage{url}		
\usepackage{color}
\usepackage{hyperref}
\usepackage{url}
\usepackage{breakurl}
\newcommand{\bburl}[1]{\textcolor{blue}{\url{#1}}}

\calclayout

\makeatletter
\newcommand{\monthyear}[1]{%
  \def\@monthyear{\uppercase{#1}}}
\newcommand{\volnumber}[1]{%
  \def\@volnumber{\uppercase{#1}}}
\AtBeginDocument{%
\def\ps@plain{\ps@empty
  \def\@oddfoot{\@monthyear \hfil \thepage}%
  \def\@evenfoot{\thepage \hfil \@volnumber}}
\def\ps@firstpage{\ps@plain}
\def\ps@headings{\ps@empty
  \def\@evenhead{%
    \setTrue{runhead}%
    \def\thanks{\protect\thanks@warning}%
    \uppercase{The Fibonacci Quarterly}\hfil}%
  \def\@oddhead{%
    \setTrue{runhead}%
    \def\thanks{\protect\thanks@warning}%
    \hfill\uppercase{Hypergeometric Template}}%
  \let\@mkboth\markboth
  \def\@evenfoot{%
    \thepage \hfil \@volnumber}%
  \def\@oddfoot{%
    \@monthyear \hfil \thepage}%
  }%
\footskip=25pt
\pagestyle{headings}%
}
\makeatother

\theoremstyle{plain}
\numberwithin{equation}{section}
\newtheorem{thm}{Theorem}[section]
\newtheorem{theorem}[thm]{Theorem}

\newtheorem{lem}[thm]{Lemma}

\newtheorem{proposition}[thm]{Proposition}

\begin{document}
\monthyear{Month Year}
\volnumber{Volume, Number}
\setcounter{page}{1}

\title{An Invitation to Fibonacci Digits}
\author{Justin Cheigh, Guilherme Zeus Dantas e Moura, Jacob Lehmann Duke, Annika Mauro, Zoe McDonald, Anna Mello, Kayla Miller, Steven J. Miller, Santiago Velazquez Iannuzzelli}

\address{Department of Mathematics and Statistics\\
                Williams College\\
                Williamstown, MA\\}
\email{jhc5@williams.edu}

\address{Department of Mathematics and Statistics\\
    Haverford College\\
    Haverford, PA,}
\email{zeusdanmou@gmail.com}

\address{Department of Mathematics and Statistics\\
                Williams College\\
                Williamstown, MA\\}
\email{jl34@williams.edu}

\address{Department of Mathematics\\
    Stanford University\\
    Stanford, CA}
\email{amauro@stanford.edu}

\address{Department of Mathematics and Statistics\\
                Boston University\\
                Boston, MA\\}
\email{zmcd@bu.edu}

\address{Lanesborough Elementary School\\
            Lanesborough, MA\\}
\email{amello@lanesboroughschool.org}

\address{Department of Mathematics and Statistics\\
                Williams College\\
                Williamstown, MA\\}
\email{km30@williams.edu}

\address{Department of Mathematics and Statistics\\
                Williams College\\
                Williamstown, MA\\}
\email{sjm1@williams.edu}

\address{Department of Mathematics\\
                University of Pennsylvania\\
                Philadelphia, PA\\}
\email{smvelian@sas.upenn.edu}

\thanks{This research was supported by NSF Grant DMS1947438 and Williams College. We are grateful to numerous colleagues over the years for many pedagogical conversations, to the students of the Mount Greylock Regional School District (especially many classes at Lanesborough Elementary), to Karyn McLellan and the referee for an incredibly detailed reading of an earlier draft and many suggestions which improved the exposition. We also thank OIT at Williams College in general, and Matthew Baya and Tattiya Maruco in particular, for computer support for this and related projects over the years.}

\begin{abstract} The purpose of this short note is to show the interplay between math outreach and conducting original research, in particular how each can build off the other.
\end{abstract}

 \maketitle

\section{Fibonacci Outreach in Elementary Schools}

This topic is extremely near and dear to our hearts: math outreach. Specifically, engaging students and educators of all ages and levels. There are many items we could discuss; in the interest of space we focus on an activity that three of us have done numerous times at several grades in the Mount Greylock Regional School District (especially in the Lanesborough Elementary School). We've gone all the way down to kindergarten classes and all the way up to sixth grade, and it's absolutely amazing to see how engaged the kids can be, and how energized the teachers become. After describing some of our classroom activities, we pivot to the other side of the coin: how outreach activities lead to new research. This is an excellent example of multi-tasking, which is especially important given all the demands on our time.


It's valuable to learn how to look at problems the right way and from the right perspective. Kids are good at solving problems that are similar to something they have seen, but what happens when they meet something new? We have found students of all ages rise to the challenge, and love the sense of discovery and ownership that develops. We thus spend a lot of time talking about how important it is to look at items the right way. For example, consider the two Rubik's ``cubes'' below (see Figure \ref{fig:twocubes}).

\begin{figure}[h]
\begin{center}
\scalebox{.8}{\includegraphics{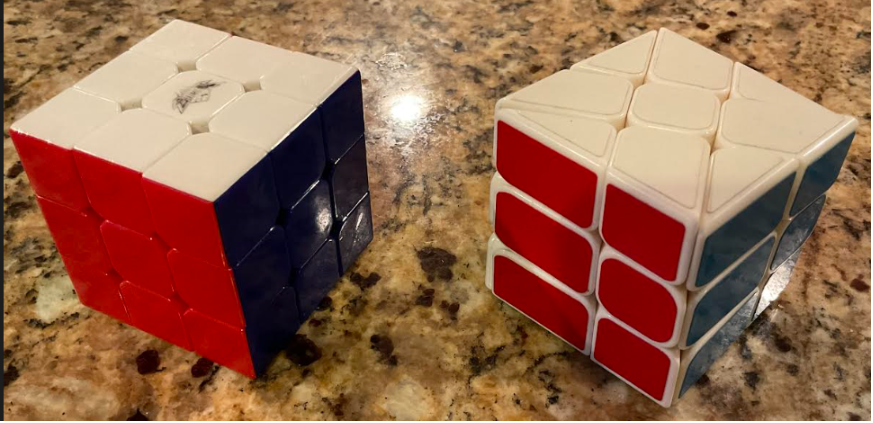}}
\caption{\label{fig:twocubes} A standard and a very non-standard Rubik's cube!}
\end{center}\end{figure}

At first glance the two look very different. The standard one has six sides of different colors, each with nine pieces, while the other also has six differently-colored sides, only two have nine pieces while the other four are $2 \times 3$ (unequally sized pieces) rather than $3 \times 3$. Yet, amazingly, if we look at the second cube from the right perspective (see Figure \ref{fig:twocubes2}) we see it's almost the same as the first. Specifically, if we hold it at a 30 degree angle the four sides that are $2 \times 3$ become sides that are $3 \time 3$! For example, the front face is not a red side, but rather a red side running into blue. This is a great example, easy for students to see, of how valuable it is to find the right vantage point for a problem: in viewing the non-standard Rubik's cube like the original, we can use the known solution to help us solve this new problem.

\begin{figure}[h]
\begin{center}
\scalebox{.8}{\includegraphics{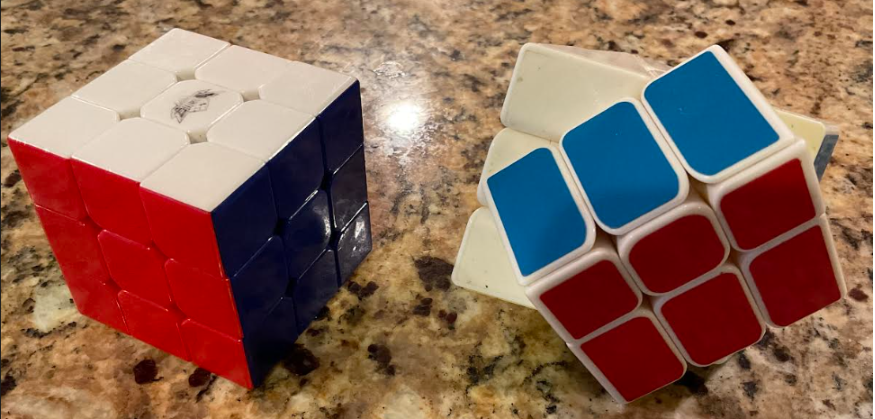}}
\caption{\label{fig:twocubes2} Viewing the non-standard cube from an angle reveals a similarity to the standard cube.}
\end{center}\end{figure}


The following activity is related to the lesson from the cube. While when we go to schools to play math games we cannot lie, we can of course deliberately mislead (so long as we come clean in the end). Here is what we call the \emph{I Love Rectangles Game}. \\ \

\noindent \textbf{I Love Rectangles Game:} \emph{You have infinitely many squares of distinct size: one each of a $1\times 1$, $2\times 2$, $3\times 3$, and $4\times 4$, and so on (see Figure \ref{fig:squares}). The goal is to put as many of them down on the table as possible, subject to some constraints. You can't put one square on top of another, and we just love rectangles so we will not consider any other shape. Thus, however many squares are on the table, they must be connected and form a rectangle. How many squares can be so placed?}

\begin{figure}[h]
\begin{center}
\scalebox{.8}{\includegraphics{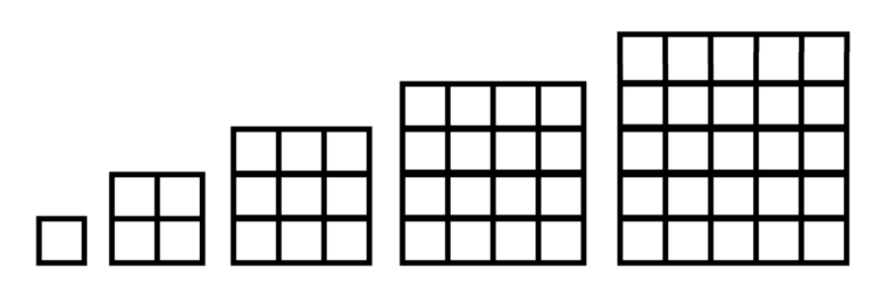}}
\caption{\label{fig:squares} The first few pieces for the \emph{I Love Rectangles} game.}
\end{center}\end{figure}

\ \\


When we go to classrooms, we put one of the squares down, say we have a rectangular shape (as every square is a rectangle), and we're done. Students are then offered \$20 if they can place two down. They are very excited and engaged! They try many different approaches, physically moving the squares along and trying to combine. We normally break the class (depending on size) into groups of 2 to 4 students and we bring in plastic models / LEGO models so they can get their hands on the math.


After a while, kids realize that the game is rigged and it is impossible to choose any two squares from those in front of them to make a rectangle. (If we've gone to their classroom many times before this activity, they are frequently suspicious of any game we introduce.) Let's make it more interesting now. We're cheap, so let's add the fewest pieces possible. We have to give another piece because as you cannot have two squares form a rectangle if they have unequal sides. We'll add just one more square, and as we're cheap we'll do the smallest size square we can: another $1 \times 1$. So now we have two of those.

The kids now need to figure out what to do. We have yet to go to a classroom where students haven't realized that they must start by putting the two $1 \times 1$ squares together; if we don't use the new piece, we're in the old case and we know there is nothing that can be done. This is a great way to talk about using the givens of a problem.

The kids very quickly start building at this point. They put down a $1 \times 1$, put down another $1 \times 1$, then a $2\times 2$, a $3\times 3$, a $5\times 5$, and so on, and they start seeing the pattern emerge. It's amazing -- we've never done this where even one group has failed to discover the Fibonacci numbers. They realize they can't use the $4\times4$ and after the $5\times 5$, the next one they need to use is the $8\times 8$ (see Figure \ref{fig:squares2}). As young as kindergarten, they're figuring out the patterns and it's just wonderful to see them doing this. What's nice is that there are many different paths we can take, depending on their grade and interest. For more advanced classes, we can talk about proofs of sums of squares and products of Fibonacci numbers and much, much more.\footnote{For example, the tiling gives a geometric proof that $F_1^2 + F_2^2 + \cdots + F_n^2 = F_n F_{n+1}$ by calculating the area of the $F_n \times F_{n+1}$ rectangle two different ways: length times width, and the sum of the area of the squares making it up.}

\begin{figure}[h]
\begin{center}
\scalebox{.8}{\includegraphics{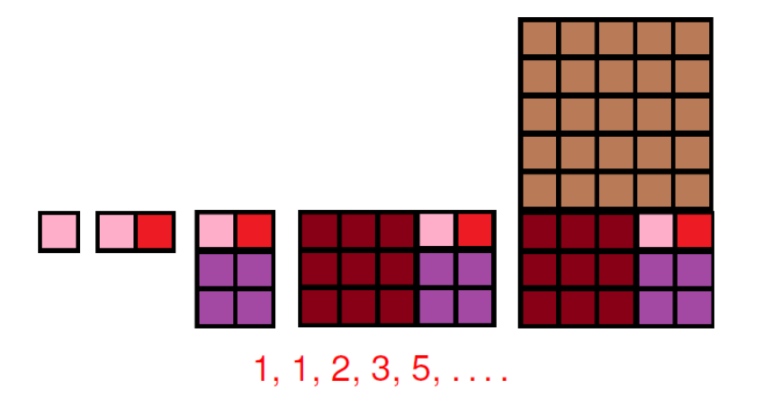}}
\caption{\label{fig:squares2} The consequences of adding an additional $1 \times 1$ piece for the \emph{I Love Rectangles} game.}
\end{center}\end{figure}

After our success in building arbitrarily large rectangles, we talk about what freedom is allowed here. After the first tile, we actually have four places to put the second $1 \times 1$ square, and then after that from that point on, we always have exactly two locations. How should we choose? This conversation flows to discussing the Fibonacci spiral. There's a beautiful video \emph{Nature by Numbers} (see \bburl{https://youtu.be/kkGeOWYOFoA}), which shows many wonderful applications of the Fibonacci sequence.


Figure \ref{fig:kidsplay} shows some kids playing and discovering the Fibonacci numbers. You can see them talking to each other as they try to put together the various pieces to make a square. This is step one in our visit. Then we ask them what they observe. They write the numbers on the board and they start trying to determine patterns. Very quickly they notice that the terms go odd, odd, even. Usually someone then asks whether it's always going to be odd, odd, even. Frequently several kids notice that each one seems to be the sum of the previous two. This is an excellent opportunity to introduce the idea of mathematical notation (or even better, \emph{good} mathematical notation). For many, this is the first time they've had to come up with a \emph{name} for a sequence, and a way to refer to a general term.

\begin{figure}
\begin{center}
\scalebox{1}{\includegraphics{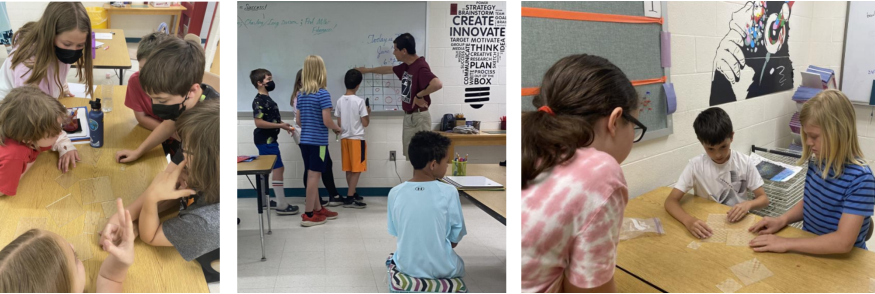}}
\caption{\label{fig:kidsplay} Fourth graders at Lanesborough Elementary discovering the Fibonacci numbers.}
\end{center}\end{figure}

\ \\

Here are two representative comments from recent fourth graders.

\begin{itemize}

\item The pattern goes even even odd for every sequence if you start with 1, 1, or any odd numbers. If you start with an even number second like 3, 2, the sequence is shaken up so it goes odd, even, odd, but after that it goes back to the other odd odd even sequence. If you start with two even numbers, it will always be even. If you make a diagram it goes in a shell that matches a nautilus shell. --William.\\ \

\item Why is the Fibonacci sequence everywhere? What is the purpose of the Fibonacci sequence? Why do the Fibonacci sequence numbers sound so good on instruments? Can the Fibonacci sequence be found in other shapes than a swirl? Can Fibonacci be found in technology like phones, laptops, or TVs? --Maddie. \\ \

\end{itemize}




\ \\

These two activities worked well with older elementary school kids. We've had competitions to see who can calculate the largest Fibonacci number: it's a Fibonacci race! It's interesting to see what strategy is best to adopt: do you just try to do it as fast as possible, but with the risk that once you make a mistake, odds are you're out as it's unlikely you'll correct, or do you go a little bit slower and be more careful? This is a great opportunity to talk to kids about spreadsheets and explore how you would use something like that technology as a tool to assist you if you want to do really large calculations.  We then launch a Google Sheet and introduce some basic commands, such as how to fill down and how to define formulas. For most of the kids, this is the first time they've really thought about formulas and gathering data.

\section{Research inspired by Fibonacci Outreach}

Each summer Williams College provides research opportunities to college mathematics students through the SMALL program (the letters come from the first letters of the last names of the founding five professors; see \bburl{http://math.williams.edu/small/} for more information). It's run for almost 40 years, and in addition to research the students are frequently involved in outreach activities.

In Summer '22, S. Miller invited his current SMALL research students to work on math activities to be used with students in elementary schools. As an exercise to get them thinking from the perspective of young kids, Miller asked the group to tell him about patterns they noticed in the Fibonacci numbers upon first glance. He expected his students to comment on the well-known pattern of the Fibonacci numbers alternating odd, odd, even, but was instead (pleasantly!) surprised when one of his students, Cheigh, asked the following: does every Fibonacci number have either a $1, 2, 3, 5, $ or $8$ as one of its digits? Of course, $\{1, 2, 3, 5, 8\}$ must be a subset of the minimal set of digits needed to have an element in each Fibonacci number, since these digits are exactly the single-digit Fibonacci numbers. The group became curious and discussed several generalizations.\\ \

\begin{itemize}

\item For a given base $B$, for what fixed subsets of $\{0, \dots, B-1\}$  must all Fibonacci numbers have at least one member as a digit? \\ \

\item Reconsider the previous problem, but now allow finitely many exceptions. In other words, for what subsets will all sufficiently large Fibonacci numbers have at least one of these as a digit? \\ \

\item While problems such as the above are often very difficult, as they are deterministic, frequently a random probabilistic model is easier to analyze and can build intuition (see for example \cite{kontlag}). What do such models say about these questions? \\ \

\end{itemize}

As a professor teaching students how to participate in math outreach (as well as doing research themselves), Miller was careful not to involve himself in investigating these questions, despite his interest in the problems! The students quickly came up with a partial solution to the original question posed by Cheigh. We sketch their solution below, and provide a complete proof in Section \ref{sec3}.

We attack this problem by brute force via computer calculations; we label the terms of the Fibonacci sequence by $F_1 = 1, F_2 = 1, F_3 = 2$ and so on. The last $d$ digits of the Fibonacci sequence are periodic (we'll prove this shortly in Lemma \ref{lem:pisano}, but for now let's just assume this result and see explore the consequences\footnote{If we look at the Fibonacci numbers modulo 10, we get $\{$1, 1, 2, 3, 5, 8, 3, 1, 4, 5, 9, 4, 3, 7, 0, 7, 7, 4, 1, 5, 6, 1, 7, 8, 5, 3, 8, 1, 9, 0, 9, 9, 8, 7, 5, 2, 7, 9, 6, 5, 1, 6, 7, 3, 0, 3, 3, 6, 9, 5, 4, 9, 3, 2, 5, 7, 2, 9, 1, 0, 1, 1, 2, 3, 5, $\dots\}$, and we see that the last digit is periodic with period 60; we leave to the reader the fun exercise of finding the period of the last two digits!}). Hence it suffices to check that finitely many Fibonacci numbers include an element from a particular set as one of their last $d$ digits. So far, we've proven that if you omit a 6, every Fibonacci number has one of the remaining digits, 0, 1, 2, 3, 4, 5, 7, 8, or 9. This means that every Fibonacci number contains an element of the set $\{0, 1, 2, 3, 4, 5, 7, 8, 9\}$. If we instead omit a 4, we have a similar result. If we omit a 2, for all Fibonacci numbers sufficiently large, this argument still works. Note it cannot be true for all Fibonacci numbers as $F_2 = 2$, but this is the only problem -- thus 3 is sufficiently large!



Unfortunately, though, some very large Fibonacci numbers don't have a $1, 2, 3, 5$ or $8$ in their last $d$ digits. Take, for example $F_{21} = 10946$. The first time an element of $\{1, 2, 3, 5, 8\}$ appears is in the fifth-from-last position. For all Fibonacci numbers smaller than $F_{21}$, the furthest a $1, 2, 3, 5$ or $8$ appears is in the third-from-last position. This means that under our brute force method, $F_{21}$ forces us to increase our search from checking the last three digits of every Fibonacci through $F_{15\cdot{10}^2}$ to checking the last five digits through $F_{15\cdot{10}^4}$ as per Theorem \ref{thm:jarden} below. This issue of expanding searches persists with higher Fibonacci numbers: the last 19 digits of $F_{300}$, \begin{center} $F_{300} \ = \ 22223224462942044552973989346190996720666693\textbf{9096499764990979600}$, \end{center} do not contain a $1, 2, 3, 5$ or $8$.

Of course, we see that lots of digits of $F_{300}$ are in our set $\{1, 2, 3, 5, 8\}$, and thus it's still possible that every Fibonacci number has at least one of these as a digit, or if not, that at least all sufficiently large numbers do. In the final section we give some results on related problems, and encourage the reader to explore this and other questions (for example, what fraction of the time is $d$ a digit of $F_n$ as $n \to \infty$?) and let us know what you find!

\section{Proofs}\label{sec3}

We first quote a result due to Jarden concerning the periodic behavior of Fibonacci numbers.

\begin{theorem}[Jarden \cite{jarden}]
\label{thm:jarden}
The last $n\geq 3$ digits of the Fibonacci numbers repeat every $15 \cdot 10^{n-1}$ times.
\end{theorem}

Rather than prove the above, we'll instead state and prove a simpler, well-known result, and in the interest of encouraging exploration leave extensions to the reader.

\begin{lem}\label{lem:pisano} For any fixed modulus $m$, the Fibonacci numbers modulo $m$ are periodic with period at most $m^2+1$. The period is denoted $\pi(m)$, and called the Pisano period of the Fibonaccis modulo $m$. \end{lem}

\begin{proof} The proof is a simple application of the Pigeonhole Principle, and can be found in many texts.\footnote{Also called the Box Principle, it states that if you have $N+1$ objects and each is assigned to one and only one box, then at least one box gets at least two elements. Many problems can be successfully resolved by such arguments; see for example Lectures 11 and 12 of the eighth named author's problem solving class from 2017, \bburl{https://web.williams.edu/Mathematics/sjmiller/public_html/331Sp17/index.htm} (in particular \bburl{https://youtu.be/xx8f8z8zzfQ} and \bburl{https://youtu.be/AA0ajr-foUw}), Appendix A.4 of \cite{MT-B}, and a plethora of problems that can be found from a simple web search.} Consider the Fibonacci numbers modulo $m$ in pairs: $$(F_0 \bmod m, F_1 \bmod m), \ \ \ (F_1 \bmod m, F_2 \bmod m), \ \ \ (F_2 \bmod m, F_3 \bmod m), \ \ \ \dots.$$ Each of these pairs must be in the set $\{(a, b): 0 \le a, b \le m-1\}$, and clearly this set has exactly $m^2$ elements. If we look at the first $m^2+1$ pairs of Fibonacci numbers modulo $m$, at least two of them must be the same. Thus there are indices $i < j$ in $\{0, 1, \dots, m^2\}$ such that $$(F_i \bmod m, F_{i+1} \bmod m) \ = \ (F_j \bmod m, F_{j+1} \bmod m).$$  As each Fibonacci number is the sum of the previous two, we have $$F_{i+2} \bmod m  \ = \  (F_{i+1} + F_i) \bmod m \ = \ (F_{j+1} + F_j) \bmod m \ = \ F_{j+2} \bmod m.$$ Therefore not only do we have $F_i = F_j \bmod m$ and $F_{i+1} = F_{j+1} \bmod m$ but we also have $F_{i+2} = F_{j+2} \bmod m$; a similar argument shows that we can keep marching down and $F_{i+3} = F_{j+3} \bmod m$ and so on. We can also extend this argument backwards to $F_{i-1} = F_{j-1} \bmod m$, and so on and so on. Thus the sequence is periodic with period a divisor of $j-i$, which is at most $m^2+1$ (the period is this large if $i$ and $j$ are as far apart as possible, with $i = 0$ and $j = m^2$.
\end{proof}

It's an interesting question to determine how the period depends on $m$. When $m=2$ the period is 3 as $0, 1, 1, 2, 3, 5, 8, 13, 21, \dots$ modulo 2 become $0, 1, 1, 0, 1, 1, 0, 1, 1, \dots$, and we see the odd, odd, even pattern remarked on earlier. For $m=3$ we find the Fibonacci numbers reduce to $$0,\ 1,\ 1,\ 2,\ 0,\ 2,\ 2,\ 1,\ 0,\ 1,\ 1,\ 2,\ 0,\ 2,\ 2,\ 1,\ 0,\ 1,\ 1,\ 2,\ 0,\ 2,\ 2,\ 1,\ 0,\ 1,\ 1,\ \dots$$ and thus the period is 8. Note trying to determine the period as a function of $m$ is a great exercise; one can use this as an opportunity to teach spreadsheets (Microsoft Excel, Google Sheets) by having the first column be the Fibonacci numbers and subsequent ones being reductions modulo various numbers. What data should we gather? Should we look at all $m$ one at a time, doing all of them in order? Should we restrict to $m$ prime? Or perhaps we take $m$ to be powers of a fixed prime? For example, the periods are 3, 6 and 12 for the first three powers of 2 (namely 2, 4 and 8) while it's 8, 24 and 72 for the powers of 3 (3, 9 and 27). Do these patterns continue for these powers? Do they hold in general? Much is known about these periods (see for example \cite{HM, Wa, We, Wr}); we are deliberately not stating the full results here to encourage exploration, as with an open mind you might ask something new! \\ \

Returning to our main problem, we explore the last $n$ digits. By Theorem \ref{thm:jarden}, if for some $n > 3$ the first $15\cdot 10^{n-1}$ Fibonacci numbers contain a digit in $S$, then all Fibonacci numbers contain a digit in $S$. The results of this are shown in Table \ref{comps}.

\begin{table}[h]
    \centering
    \begin{tabular}{c|c|c}
      $S$   &  Do first $15\cdot 10^{n-1}$  terms contain a digit in $S$? & $n$ checked\\
      \hline

       $\{0, \dots, 9\} \setminus \{1\}$\   &  \ \ \ \ no (disregarding $F_1, F_2 = 1$) &\ 8\\

       $\{0, \dots, 9\} \setminus \{2\}$\   &  \ yes (disregarding $F_3 = 2$) &\ 3\\

       $\{0, \dots, 9\} \setminus \{3\}$\   &  no (disregarding $F_4 = 3$) &\ 8\\

       $\{0, \dots, 9\} \setminus \{4\}$\   &  \hspace{-9.5em}yes & \ 3\\

      $\{0, \dots, 9\} \setminus \{5\}$\  &\ \ \ \ \ \ \ \ \ \ \ \ \ no (disregarding $F_5 = 5$, $F_{10} = 55$) &\ 8\\

      $\{0, \dots, 9\} \setminus \{6\}$\    &  \hspace{-9.5em}yes & \ 5 \\

      $\{0, \dots, 9\} \setminus \{7\}$\ & \hspace{-9.75em}no & 10\\

      $\{0, \dots, 9 \} \setminus \{8\}$\ & no (disregarding $F_6 = 8$) &\ 8\\

      $\{0, \dots, 9 \} \setminus \{2,4\}$ & no (disregarding $F_3 = 2$) & 10\\

      $\{0, \dots, 9 \} \setminus \{2,6\}$ & no (disregarding $F_3 = 2$) & 10\\

      $\{0, \dots, 9 \} \setminus \{4,6\}$ & \hspace{-9.75em}no & 10\\

    \end{tabular}
    \vspace{1em}
    \caption{Result from computing $15\cdot 10^{n-1}$ terms of Fibonacci sequence mod $10^{n-1}$ for various $S$. We say ``disregarding'' to mean that our result holds for all Fibonacci numbers except for the choice of Fibonacci number labeled disregarded. Notice that the only disregarded Fibonacci numbers are all single-digit Fibonacci numbers save for 55, which is the only multi-digit Fibonacci number whose decimal digits are all the same (see \cite{L}).}
    \label{comps}
\end{table}

As a result of the computations in Table \ref{comps}, we can conclude that all Fibonacci numbers contain:
\begin{itemize}
    \item a digit in $\{0,\dots, 9\} \setminus \{6\}$,
    \item a digit in $\{0,\dots, 9\} \setminus \{4\}$,
    \item a digit in $\{0,\dots, 9\} \setminus \{2\}$ (disregarding $F_2=2$).
\end{itemize}

\ \\

We also have an alternate proof in the case $S = \{1, \dots, 9\} \setminus \{6\}$; we encourage you to modify it for the other digits and compare with the proof in \cite{L}.

\begin{proposition}
    No Fibonacci number can be written in base 10 using only the digit $6$.
\end{proposition}

\begin{proof}
    The Fibonacci sequence mod $2^5$ is periodic, with period at most $2^{10}$ by Lemma \ref{lem:pisano}. By checking the first $2^{10}$ Fibonacci numbers, we discover that none of them are congruent to $10 \bmod{2^5}$; thus, no Fibonacci numbers are congruent to $10 \bmod{2^5}$. Since 66666 is equivalent to 10 modulo $2^5$ and no Fibonacci number is equivalent to 10 modulo $2^5$, we see that the only possible Fibonacci numbers ending with all 6's are $6666$, $666$, $66$, and $6$. A quick calculation shows none of these are Fibonacci numbers, and we have therefore proved that no Fibonacci number can be written using only the digit $6$.
\end{proof}

\ \\

\ \\

\begin{thebibliography}{109901}

\bibitem[HM]{HM}
\newblock B. H. Hannon and W. L. Morris, \emph{Tables of Arithmetical Functions Related to the Fibonacci Numbers. Report ORNL-4261}, Oak Ridge National Laboratory, Oak Ridge, Tennessee, June 1968.

\bibitem[Ja]{jarden}
\newblock D. Jarden, \emph{On the Periodicity of the Last Digits of the Fibonacci Numbers}, The Fibonacci Quarterly \textbf{1} (1963), no. 4, 21--22. 

\bibitem[KL]{kontlag}
\newblock A. Kontorovich and J. Lagarias, \emph{Stochastic Models for the $3x + 1$ and $5x + 1$ Problems}, in: The Ultimate Challenge: The $3x+1$ Problem, edited by J. Lagarias, American Mathematical Society (2010), 131--188.  \bburl{https://arxiv.org/pdf/0910.1944.pdf}.

\bibitem[L]{L}
\newblock F. Luca, \emph{Fibonacci and Lucas Numbers With Only One Distinct Digit}, Portugaliae Mathematica \textbf{57} (2000), Fasc. 2, 243--254.

\bibitem[MT-B]{MT-B}
\newblock S. J. Miller and R. Takloo-Bighash, \emph{An Invitation to Modern Number Theory}, Princeton University Press, Princeton, NJ, 2006, 503 pages.

\bibitem[Wa]{Wa}
\newblock D. D. Wall, \emph{Fibonacci Series Modulo $m$}, American Mathematical Monthly \textbf{67} (1960), 525--532.

\bibitem[We]{We}
\newblock Eric W. Weisstein, \emph{Pisano Period}, From MathWorld--A Wolfram Web Resource. \bburl{https://mathworld.wolfram.com/PisanoPeriod.html}.


\bibitem[Wr]{Wr}
\newblock J. W. Wrench, \emph{Review of B. H. Hannon and W. L. Morris, Tables of Arithmetical Functions Related to the Fibonacci Numbers}, Math. Comput. \textbf{23} (1969), 459--460.

\end{thebibliography}
\end{document}